\documentclass[10pt, reqno]{amsart}
\usepackage{amsmath,amsfonts,amsthm,bm} 
\numberwithin{equation}{section}

\usepackage{amsfonts}
\usepackage{amssymb}
\usepackage{latexsym}

\usepackage{tikz}
\usepackage{graphicx,graphics,epsfig}
\usepackage{epstopdf}

\usepackage[]{enumerate}
\usepackage{verbatim}
\usepackage{bbm}

\usepackage[]{enumerate}

\usepackage[sc]{mathpazo}          
\usepackage{eulervm}               
\usepackage[scaled=0.86]{berasans} 
\usepackage[scaled=1]{inconsolata} 
\usepackage[T1]{fontenc}

\usepackage[final]{hyperref}
\hypersetup{colorlinks=true, linkcolor=blue, anchorcolor=blue, citecolor=red, filecolor=blue, menucolor=blue, pagecolor=blue, urlcolor=blue}
\DeclareMathOperator{\sech}{sech}

\newcommand{\abs}[1]{\vert #1 \vert}

\newcommand{\Bigabs}[1]{\Bigl\vert #1 \Bigr\vert}

\newcommand{\norm}[1]{\left\Vert #1 \right\Vert}

\newcommand{\N}{\mathbb{N}}
\newcommand{\Z}{\mathbb{Z}}
\newcommand{\R}{\mathbb{R}}

\newcommand{\angles}[1]{\langle #1 \rangle}

\DeclareMathOperator{\sgn}{sgn}
\DeclareMathOperator{\supp}{supp}

\newtheorem{theorem}{Theorem}

\newtheorem{lemma}{Lemma}
\newtheorem{corollary}{Corollary}

\theoremstyle{definition}

\theoremstyle{remark}
\newtheorem{remark}{Remark}

\title[Dispersive estimates for 
 water wave equations]{Dispersive estimates for  linearized
 water wave type equations
	in $\R^d$}

  \author [T. Deneke] { Tilahun Deneke }

\author[T. T. Dufera] {Tamirat T. Dufera}

  \author[A. Tesfahun]{Achenef Tesfahun}

\address{Department of Mathematics \\
Nazarbayev University \\
Qabanbai Batyr Avenue 53 \\
010000 Nur-Sultan \\
Republic of Kazakhstan}

\email{achenef@gmail.com}

\address{Department of Mathematics
\\
Adama University of Science and Technology
\\
Ethiopia}
 
   \email{tamirat.temesgen@astu.edu.et, tilahundeneke8@yahoo.com}


\begin{document}

\begin{abstract} 
 We derive a  \(L^1_x (\R^d)-L^{\infty}_x (\R^d)\) decay estimate of order $\mathcal O \left( t^{-d/2}\right)$ for the linear propagators
$$
\exp \left( {\pm it  \sqrt{ |D|\left(1+ \beta |D|^2\right)  \tanh |D |  }   }\right) ,   \qquad \beta \in \{0, 1\}. \quad D= -i\nabla,
$$
with a loss of $3d/4$ or $d/4$ -- derivatives in the case $\beta=0$ or $\beta=1$, respectively.
These linear propagators are known to be associated with the linearized
water wave equations, where the parameter $\beta$  measures surface tension effects.
As an application we prove low regularity well-posedness for a  Whitham--Boussinesq type system in $\R^d$, $d\ge 2$. This generalizes a recent result by Dinvay, Selberg and the third author where they proved low regularity well-posedness in $\R$ and $\R^2$.

\end{abstract}

\maketitle

 \section{Introduction}
 In this paper, we derive a $L^1_x (\R^d)-L^{\infty}_x (\R^d)$ time-decay estimate for the linear propagators
$$
\mathcal S_{m_\beta}(\pm t) := \exp\left(\mp it  m_\beta (D) \right) ,
$$
where
$$
m_\beta (D)=\sqrt{ |D|\left(1+ \beta |D|^2\right)  \tanh  |D |  }
$$
with  $\beta \in \{0, 1\}$ and $D=-i\nabla$.
The pseudo-differential operator  $m_\beta (D)$ appears in linearized
water wave type equations. The cases $\beta=0$ and $\beta=1$ correspond respectively to purely gravity waves and capillary-gravity waves. 

For instance, consider the Whitham equation without or with surface tension
(see e.g., \cite{DMDK2017, RK2017})
\begin{equation}\label{Whit}
u_t+L_\beta u_x+ uu_x=0,   \qquad  (\beta \ge 0),
\end{equation}
where $u:\R\times \R \rightarrow \R$, and the non-local operator $L_\beta$ is related to the dispersion relation of the (linearized) water waves system and is defined by
$$L_\beta:=L_\beta(D) = \sqrt{ \left(1+ \beta |D|^2  \right)  K(D) }$$
with
$$K(D) =\frac{\tanh  |D|}{ |D|} \qquad (D= -i\nabla=-i\partial_x).$$
The linear part of \eqref{Whit} can be written as
\begin{equation}\label{Whit-lin}
iu_t -  \widetilde m_\beta(D) u=0,
\end{equation}
where
$$\widetilde m_\beta (D)=D L_\beta(D)= \frac {D} {|D|} m_\beta (D) .$$
In terms of Fourier symbols we have $\widetilde m_\beta(\xi)= \sgn(\xi) m_\beta(\xi)$.
So the solution propagator for \eqref{Whit-lin} is given by $\mathcal S_{\widetilde m_\beta}(t)=\exp\left( - it  \widetilde m_\beta (D) \right) $.
 In fact, both $\mathcal S_{\widetilde m_\beta}(t)$ and $\mathcal S_{ m_\beta}(t)$ satisfy the same $L^1_x (\R)-L^{\infty}_x (\R)$ time-decay estimate.

As another example, consider the full
dispersion Boussinesq system (see e.g., \cite{KLPS18, D2013}) 
\begin{equation}
    \label{FD2d}
    \left\lbrace
    \begin{array}{l}
    \eta_t+  L_\beta^2 \nabla\cdot {\bf v}+ \nabla \cdot(\eta {\bf v})=0 \\
    {\bf v}_t+\nabla \eta+\nabla |{\bf v}|^2=0,
 \end{array}\right.
    \end{equation}
    where $$
    \eta:\R\times \R \rightarrow \R, \qquad \mathbf v:\R^d\times \R \rightarrow \R^d.$$
    This system
describes the evolution with time of surface waves of
a liquid layer, where 
  $\eta$ and $\mathbf v$
denote the surface elevation
and the fluid velocity, respectively. One can derive an equivalent system of \eqref{FD2d}, by diagonalizing its linear part. Indeed, define
\begin{align*}
w_\pm=\frac{\eta \mp i L_\beta  \mathcal R \cdot \mathbf v}{2  L_\beta },
\end{align*}
where $(\eta, \mathbf v)$ is a solution to \eqref{FD2d}, and 
 $ \mathcal R=|D|^{-1}\nabla $ is the Riesz transform. Then the linear part of the system \eqref{FD2d} transforms to 
 \begin{equation}
    \label{FD2d-lin}
    i \partial_t w_\pm  \mp m_\beta(D) w_\pm=0
    \end{equation}
 whose corresponding solution propagators are
$
\mathcal S_{m_\beta}( \pm t) = \exp\left(\mp  it  m_\beta (D) \right)  .
$
    
    As a third example, consider the Whitham--Boussinesq type system (see e.g., \cite{Cr18, DDT22, Dy19, DDK19}) 
\begin{equation}
\label{wtbsq}
\left\{
\begin{aligned}
  \partial_t \eta +  \nabla \cdot  \mathbf v & = -  K\nabla \cdot (\eta \mathbf v),
  \\
  \partial_t \mathbf v + L_\beta^2 \nabla \eta  &=   -  K  \nabla ( |\mathbf v |^2/2).
\end{aligned}
\right.
\end{equation} 
 Again, by defining the new variables
\begin{equation}\label{diag-wtsbq}
u_\pm=\frac{ L_\beta \eta \mp i  \mathcal R\cdot \mathbf v}{2 L_\beta },
\end{equation}
    we see that the linear part of \eqref{wtbsq} transforms to 
 \begin{equation}
    \label{wtbsq-lin}
    i \partial_t u_\pm  \mp m_\beta(D) u_\pm=0
    \end{equation}
 whose solution propagators are again
$
\mathcal S_{m_\beta}( \pm t) = \exp\left(\mp  it  m_\beta (D) \right)  .
$

So the linear propagators 
$
\mathcal S_{m_\beta}( \pm  t) 
$  appear in all of the equations \eqref{Whit}, \eqref{FD2d} and \eqref{wtbsq}. 
Since the symbol $m_{\beta}$ is non-homogeneous, we will derive a time-decay estimate from
$$
\mathcal S_{m_\beta}( \pm  t): L^1_x (\R^d) \rightarrow  L^{\infty}_x (\R^d)
$$ 
for frequency localised functions. 
To this end, we fix a smooth cutoff function $\chi$ such that
\begin{equation*}
\chi \in C_0^{\infty}(\mathbb R), \quad 0 \le \chi \le 1, \quad
\chi_{|_{[-1,1]}}=1 \quad \mbox{and} \quad  \mbox{supp}(\chi)
\subset [-2,2].
\end{equation*}
 Set
$$
\rho(s)
=\chi\left(s\right)-\chi \left(2s\right).
 $$
 Thus, $\supp \rho= 
\{ s\in \R: 1/ 2 \le |s| \le 2 \}$. For 
  $\lambda \in  2^\Z$ we set $\rho_{\lambda}(s):=\rho\left(s/\lambda\right)$
 and define the frequency projection $P_\lambda$ by
\begin{align*}
\widehat{P_{\lambda} f}(\xi)  = \rho_\lambda(|\xi|)\widehat { f}(\xi) .
 \end{align*}
 Sometimes, we write $f_\lambda: =P_\lambda f$.

 \medskip
\noindent \textbf{Notation}. 
For any positive numbers $a$ and $b$, the notation $a\lesssim b$ stands for $a\le cb$, where $c$ is a positive constant that may change from line to line. Moreover, we denote $a \sim b$  when  $a \lesssim b$ and $b \lesssim a$.
We also set $\angles{x}:= \left(1 +|x|^2\right)^{1/2}.$

 For $1\le p\le\infty$, $L^p_x(\mathbb R^d)$ denotes the usual Lebesgue space and for $s\in\R$, $H^s(\mathbb R^d)$ is the $L^2$-based Sobolev space with norm $\|f\|_{H^s}=\|\angles{D}^s f\|_{L^2}$.
If $T>0$ and $1\le q \le\infty$, we define the spaces $L^q\big((0,T) : L^r (\R^d)\big)$ and $L^q\big( \mathbb R : L^r (\R^d)\big)$ respectively through the norms
$$
\|f\|_{L^q_TL^r_x} = \left( \int_0^T \|f(\cdot,t)\|_{L^r_x}^q dt \right)^{\frac1q} \quad  \textrm{and} \quad \|f\|_{L^q_tL^r_x} = \left( \int_{\mathbb R} \|f(\cdot,t)\|_{L^r_x}^p dt \right)^{\frac1q} \, ,
$$
when $1 \le q < \infty$, with the usual modifications when $q=+\infty$.

 \vspace{2mm}
 
Our first result is as follows:
\begin{theorem}[Localised dispersive estimate] \label{thm-dispest}
Let $\beta \in \{0, 1\}$,  $d\ge 1$ and $\lambda\in 2^\Z$.
Then 
\begin{align}
\label{dispest}
\| \mathcal S_{m_\beta}(\pm t) f_{\lambda}  \|_{L^\infty_x(\R^d)} &\lesssim c_{\beta,d} (\lambda) \ |t|^{-\frac d2}  \| f\|_{L_x^1(\R^d)} \end{align}
for all $f \in \mathcal{S}(\R^d)$, where 
\begin{equation}\label{cbeta}
 c_{\beta,d} (\lambda)= \lambda^{\frac d2-1}  \angles{ \sqrt{\beta}  \lambda}^{-\frac d2} \angles{\lambda}^{\frac d4+1}.
\end{equation}

 \end{theorem}
 
 \vspace{2mm}
\begin{remark} 
In view of \eqref{cbeta} the loss of derivatives (this corresponds to the exponent of $\lambda$) is $3d/4$ in the case 
$\beta=0$, whereas  the loss is $d/4$ when $\beta=1$.
 \end{remark}

Once Theorem \ref{thm-dispest} is proved, the corresponding Strichartz estimates are deduced from a classical $TT^{\star}$ argument.
\begin{theorem}[Localised Strichartz estimate]\label{thm-Str}
Let $\beta \in \{0, 1\}$,  $d\ge 1$ and $\lambda\in 2^\Z$.
Assume that the pair $(q,r)$ satisfies the following conditions:
\begin{equation}\label{endtpt}
  2< q \le \infty,  \quad 2\le  r \le \infty, \qquad \frac 2q+ \frac dr = \frac d2.
\end{equation}
Then 
 \begin{align}
\label{Strest1d}
\norm{ \mathcal S_{m_\beta}(\pm t)  f_{\lambda}}_{ L^{q}_{t} L^{r}_{ x} (\R^{d+1}) } \lesssim \left[ c_{\beta,d} (\lambda) \right]^{ \frac 2{qd}  } 
\norm{  f_{\lambda}}_{ L^2_{ x}(\R^d )} ,
\\
\label{Strest1d-inh}
\norm{ \int_0^t  \mathcal S_{m_\beta} (\pm(t-s) )F_\lambda (s) \, ds}_{ L^{q}_{t} L^{r}_{ x} (\R^{d+1}) } \lesssim  \left[ c_{\beta,d} (\lambda) \right]^{ \frac 2{qd} } 
\norm{ F_{\lambda}}_{ L_t^1L^2_{ x}(\R^{d+1} )} 
\end{align}
for all  $f \in \mathcal{S}(\R^d)$ and  $F \in \mathcal{S}(\R^{d+1})$, where $c_{\beta,d} (\lambda)$ is as in \eqref{cbeta}.

\end{theorem}

As an application of Theorem \ref{thm-Str} we prove low regularity well-posedness for the system \eqref{wtbsq} with $\beta=0$ (i.e., for purely gravity waves) in $\R^d$, $d\ge 2$. 
To this end, we complement the system \eqref{wtbsq} with initial data 
\begin{equation}\label{data-wtbsq}
\eta(0)= \eta_0 \in  H^s(\mathbb{R}^d), \qquad  \mathbf v(0)= \mathbf v_0 \in  \left( H^{s + 1/2} \left( \mathbb R^d \right) \right) ^d.
\end{equation}
\begin{theorem}
\label{theorem3d}
	Let $\beta=0$, $d\ge 2$ and $s>\frac d2-\frac 34$. Suppose that $\mathbf v_0$ is a curl-free vector field, i.e., $\nabla \times \mathbf v_0=0$, and
	$$
	\left[\lVert \eta_0 \rVert _{H^s(\R^d)}
		+ \lVert \mathbf v_0 \rVert _{ (H^{s + 1/2} (\R^d)) } \right] \le \mathcal D_0.
	$$
	Then there exists a solution
	$$
		( \eta, \mathbf v ) \in C \left([0, T];
		H^s(\mathbb{R}^d)
		\times
		\left( H^{s + 1/2} \left( \mathbb R^d \right) \right) ^d
		\right)
$$
of the Cauchy problem \eqref{wtbsq}, \eqref{data-wtbsq}, with existence time \footnote{Here we used the notation $a \pm: =a \pm \varepsilon$ for sufficiently small $\varepsilon >0$.} $T\sim \mathcal D_0^{-2-}$.
	
	Moreover, the solution is unique in some subspace of the above solution space and the solution depends continuously on the initial data.
\end{theorem}

 \begin{remark} 
 The following are known results:
 \begin{enumerate}[(i)]
 \item  Theorem 1 and Theorem 2  are proved in \cite{DST20} when $\beta=0$, $d\in \{1,2\}$ and $\lambda \gtrsim 1$.
 \item  Theorem 3 is proved in \cite{DST20} when $d=2$. In the case $d=1$ local well-posedness for $s>-1/10$ and  global well-posedness for small initial data in $L^2(\R)$ is also established in \cite{DST20}. Long-time existence of solution in the case $d=2$
 is also obtained in \cite{T2020}.
 \end{enumerate}

 \end{remark}

\vspace{2mm}
Van der Corput's Lemma will be useful in the proof of Theorem \ref{thm-dispest}.

\begin{lemma}[Van der Corput's Lemma, \cite{S93} ]\label{lm-corput}
 Assume 
$g \in C^1(a, b)$, $\psi\in  C^2(a, b)$ and $|\psi''(r)|  \ge  A$ for all $r\in (a, b)$. Then 
\begin{align}
\label{corput} 
\Bigabs{\int_a^b e^{i t  \psi(r)}   g(r) \, dr}& \le C  (At)^{-1/2}  \left[ |g(b)| + \int_a^b |g'(r)| \, dr \right] ,
\end{align}
for some constant $C>0$ that is independent of $a$, $b$ and $t$.
\end{lemma}

Lemma \ref{lm-corput} holds even if $\psi'(r)=0$ for some $r\in (a, b)$.
However, if $|\psi'(r)|> 0$ for all $r\in (a, b)$, one can use integration by parts to obtain the following Lemma. The proof may be found elsewhere but we include it here for the reader's convenience.
\begin{lemma}\label{lm-corput1}
Suppose that $g \in C^\infty_0(a,b)$ and $\psi\in C^\infty (a, b)$ with $|\psi'(r)|> 0$ for all $r\in (a, b)$. If 
  \begin{equation}\label{dervbd}
\max_{a\le r\le b}|\partial_r^j g(r)| \leq A, \qquad  \max_{a\le r\le b}  \Bigabs{\partial_r^j\left( \frac 1{\psi'(r)} \right)} \leq B
  \end{equation}
  for all $0 \le j \le \ N\in \N_0$,
  then 
\begin{align}
\label{corput'} 
\Bigabs{\int_a^b e^{i t  \psi(r)}   g(r) \, dr}& \lesssim A B^N  |t|^{-N} .
\end{align}
\end{lemma}
\begin{proof}
Let 
$$
I(t)= \int_a^b e^{i t  \psi(r)}   g(r) \, dr.
$$
For $m, N \in  \N_0 $, define
$$
S_{m, N}:=\{  (k_1, \cdots, k_N) \in \N_0: \ k_1<\cdots <k_N \le N \ \ \& \ \  k_1 + \cdots + k_N= m \}.
$$
Integration by parts yields
\begin{align*}
I(t)&=I_1(t)= -it^{-1} \int_a^b \partial_r\left( e^{i t\psi(r)} \right) \frac 1{\psi'(r)}  g(r) \, dr
\\
&= (-it)^{-1} \int_a^b e^{i t\psi(r)}  \left\{ \frac 1{\psi'(r)}   g'(r) +  \partial_r\left(\frac 1{\psi'(r)} \right)  g(r) \right\}  \, dr.
\end{align*}
Repeating the integration by parts $N$--times we get 
\begin{align*}
I(t)&=I_N(t)= (-it)^{-N}\sum_{m=0}^N  \sum_{(k_1, \cdots, k_N )\in S_{m, N}} C_{k, m, N}  \int_a^b  e^{i t\psi(r)}E_{k,m, N} (r) \, dr,
\end{align*}
where $C_{k, m, N} $ are constants  and 
$$
E_{k, m, N}(r)  = \Pi_{j=1}^N \partial_r^{k_j} \left( \frac 1{\psi'(r)} \right)  \cdot g^{(N-m)}(r) .
$$
Applying \eqref{dervbd} we obtain
$$
|E_{k,m, N}|\le \left(  \Pi_{j=1}^N B  \right)  A =A B^N 
$$
and hence 
$$
\abs{I(t)}= \abs{I_N(t)}  \lesssim  A  B ^N |t|^{-N} 
$$
as desired.
\end{proof}

 \section{Proof of Theorem \ref{thm-dispest} }
  Without loss of generality we may assume $\pm =-$ and $t > 0 $.
 Now we can write
\[
\left[
\mathcal S_{m_\beta}(- t) f_\lambda \right](x) 
= (I_{\lambda} (\cdot, t)\ast f)(x),
\]
where
\begin{equation}\label{Idef}
 I_{\lambda} (x, t)
=  \lambda^d  \int_{\R^d} e^{i \lambda x \cdot \xi+ it {m_\beta}(\lambda \xi)}  \rho(|\xi|) \, d\xi .
\end{equation}
By Young's inequality
\begin{equation}\label{younginq}
\| \mathcal S_{m_\beta}(- t)   f_\lambda  \|_{L^\infty_x(\R^d)} \le  \| I_{\lambda (\cdot, t)} \|_{L^\infty_x(\R^d)}  \|f\|_{L_x^1(\R^d)},
\end{equation}
and therefore, \eqref{dispest} reduces to proving
\begin{equation}
\label{Iest-1}
\|I_{\lambda} (\cdot, t) \|_{L^\infty_x(\R^d)}  \lesssim  c_{\beta,d} (\lambda)  t^{-\frac d2} .
\end{equation}

Observe that
$$
 t\lesssim \lambda^{-1/2} \angles{ \sqrt{\beta}  \lambda}^{-1}  \quad \Rightarrow \quad c_{\beta,d} (\lambda)  t^{-\frac d2}  \gtrsim  \lambda^d
$$
in which case 
\eqref{Iest-1} follows from the simple estimate
\begin{equation*}
\|I_{\lambda} (\cdot, t) \|_{L^\infty_x(\R^d)} 
\lesssim  \lambda^d.
\end{equation*}
We may therefore assume from now on
\begin{equation}\label{tcond}
t\gg  \lambda^{-1/2} \angles{ \sqrt{\beta}  \lambda}^{-1}.
\end{equation}

\subsection{Proof of \eqref{Iest-1} when $d=1$}
In the case $d=1$ we have
\begin{align*}
I_{\lambda} (x, t)
=\lambda \int_{\R} e^{i t\phi_{\lambda }(\xi) }  \rho(|\xi|) \, d\xi,
\end{align*}
where
$$
\phi_{\lambda}(\xi):= \lambda  \xi x/t  +  m_\beta(\lambda \xi) .$$
We want to prove 
\begin{equation}
\label{Iest-1d}
\|I_{\lambda} (\cdot, t) \|_{L^\infty_x(\R)}  \lesssim  \underbrace{\lambda^{-\frac 12} \angles{ \sqrt{\beta}  \lambda}^{-\frac 12}  \angles{\lambda}^{\frac 54} }_{=c_{\beta, 1}}t^{-\frac 12} 
\end{equation}
under condition \eqref{tcond}.

Since $\supp \rho= 
\{ \xi \in \R: 1/ 2 \le |\xi| \le 2 \}$, we can write
\begin{align*}
I_{\lambda} (x, t)
 &=\underbrace{\lambda  \int_{1/2}^2 e^{i t\phi_{\lambda }(\xi) }  \rho(\xi) \, d\xi }_{:=I^+_{\lambda} (x, t)}+ \underbrace{\lambda   \int_{1/2}^2 e^{i t\phi_{\lambda }(-\xi) }  \rho(\xi) \, d\xi}_{:=I^-_{\lambda} (x, t)}.
\end{align*}
We estimate only $I^+_{\lambda} (x, t)$ as the estimate for $ I^-_{\lambda} (x, t)$ can be derived in  exactly the same way.
Since
\begin{align}
\label{ph'}
\phi_{\lambda}'(\xi)&=  \lambda  \left[  x/t  +  m_\beta'(\lambda \xi)\right] ,
\qquad 
  \phi_{\lambda }''(\xi)=     \lambda^2 m_\beta''(\lambda \xi)
\end{align}
it follows from Lemma \ref{lm-mest} that
\begin{equation}
\label{Est-phi''-1}
|\phi_{\lambda }''(\xi) |\sim   \lambda^3 \angles{ \sqrt{\beta}  \lambda} \angles{  \lambda }^{-5/2}  \quad \text{for all} \ \xi \in [1/2,2].
\end{equation}

Now we prove \eqref{Iest-1d} by dividing the  region of integration into sets of non-stationary contribution:
 $\{ \xi: \ \phi_\lambda'(\xi)\neq 0 \}$ and stationary contribution: $\{ \xi: \ \phi_\lambda'(\xi) =0 \}$.
\subsubsection{Non-stationary contribution} Since $m_\beta'$ is positive, the non-stationary contribution occurs if either
$$
 x\ge 0 \quad  \text{or }
\quad 0<-x \ll  \angles{ \sqrt{\beta}  \lambda}  \angles{\lambda}^{-\frac12 }t \quad   \text{or } -x \gg \angles{ \sqrt{\beta}  \lambda}  \angles{\lambda}^{-\frac12 }t
$$
In this case we have
$$
|\phi_{\lambda}'(\xi)|\gtrsim   \lambda \angles{ \sqrt{\beta}  \lambda}  \angles{\lambda}^{-\frac12 } \quad \text{for all} \ \xi \in [1/2,2],
$$
where Lemma \ref{lm-mest} is also used.
Combining this estimate with \eqref{Est-phi''-1}
we get 
\begin{equation*}
\max_{ 1/2 \le \xi \le 1 }\Bigabs {\partial_\xi \left(  \frac 1{  \phi'_{\lambda}(\xi)  } \right)}  \lesssim \lambda  \angles{ \sqrt \beta \lambda}^{-1} \angles{\lambda}^{-\frac 32}  .
\end{equation*} 
Now this estimate can be combined with Lemma \ref{lm-corput1} for $N=1$ to estimate $I_{\lambda} (x, t) $ as
\begin{equation}\label{Iest1-1d}
\begin{split}
| I^+_{\lambda} (x, t) |
&\lesssim \lambda \cdot \lambda  \angles{\sqrt \beta \lambda}^{-1}\angles{\lambda}^{-\frac32}   t^{-1} 
\\
&\lesssim   \lambda^{-\frac 12}  \angles{\sqrt \beta \lambda}^{-\frac12 }\angles{\lambda}^{\frac 54}   t^{-\frac12} 
\end{split}
\end{equation}
where to get the second line we used \eqref{tcond}.

\subsubsection{ Stationary contribution }\label{1Dstationary} 
This occurs if 
$$
  0<-x \sim \angles{ \sqrt{\beta}  \lambda}  \angles{\lambda}^{-\frac12 }t.
$$
In this case we use Lemma \ref{lm-corput} and \eqref{Est-phi''-1} to obtain 
\begin{equation}\label{Iest-1dst}
\begin{split}
|I^+_{\lambda} (x, t)| &\lesssim \lambda \cdot  \left( \lambda^3 \angles{ \sqrt{\beta}  \lambda} \angles{  \lambda }^{-5/2}   t \right)^{-\frac 12}  \left[ |\rho(2)| + \int_{1/2}^2 |\rho'(\xi)| \, d\xi \right] 
\\
&\lesssim  \lambda^{- \frac 12} \angles{ \sqrt{\beta}  \lambda}^{- \frac 12} \angles{  \lambda }^{\frac 54}    t ^{- \frac 12} .
\end{split}
\end{equation}

\subsection{Proof of \eqref{Iest-1} when $d\ge 2$}

To prove \eqref{Iest-1} first observe that $I_{\lambda} (x, t)$ is radially symmetric w.r.t $x$, as it is 
 the inverse Fourier transform of the radial function $e^{it {m_\beta}(\lambda \xi)}  \rho(\xi) $. So we can write (see \cite[ B.5]{G08})
\begin{equation}\label{I-eq}
I_{\lambda} (x, t) =\lambda^d \int_{1/2}^2  e^{it m_\beta(\lambda r) } (\lambda r|x|)^{-\frac{d-2}{2}}  J_{\frac{d-2}{2}}( \lambda r |x|)   r^{d-1} \rho(r) \, dr,
\end{equation}
where $J_\alpha(r)$ is the Bessel function:
$$
J_\alpha(r)=\frac{ (r/2)^\alpha}{\Gamma(\alpha+1/2) \sqrt{\pi}} \int_{-1}^1  e^{ir s} \left(1-s^2\right)^{\alpha-1/2} \, ds \quad \text{for} \ \alpha>-1/2.
$$
  The Bessel function $J_\alpha(r)$ satisfies the following properties for $\alpha>-1/2$ and $r>0$ (See   \cite[Appendix B]{G08} and \cite{S93}):
\begin{align}
\label{Jm1}
J_\alpha (r) &\le Cr^{\alpha} ,
\\
\label{Jm2}
J_\alpha(r)& \le C r^{-1/2} ,
\\
\label{Jm3}
\partial_r \left[ r^{-\alpha} J_\alpha(r)\right] &= -r^{-\alpha} J_{\alpha+1}(r)
\end{align}
Moreover, it is
known that (see \cite[Chapter 1, Eq. (1.5)]{J81}),
\begin{equation}
\label{J0est}
r^{- \frac{d-2}2 }J_{ \frac{d-2}2}(s)= e^{is} h(s)  +e^{-is}\bar h(s)
\end{equation}
for some function $h$ satisfying the decay estimate 
\begin{equation}
\label{h-est}
| \partial_r ^k h(r)|\le C_k \angles{r}^{-\frac{d-1}2-k}  \qquad  ( k\ge 0).
\end{equation}

We use the short hand
$$ 
 m_{\beta, \lambda}(r) = m_\beta(\lambda r),  \qquad \tilde J_a(r)= r^{-a} J_a(r), \qquad \tilde\rho(r)=r^{d-1} \rho(r).$$
Hence
\begin{equation}\label{I-eqq}
I_{\lambda} (x, t) = \lambda^d  \int_{1/2}^2  e^{it  m_{\beta, \lambda}(r)}  \tilde J_{\frac{d-2}{2}}( \lambda r |x|)  \tilde  \rho(r) \, dr,
\end{equation}

We prove \eqref{Iest-1} by treating the cases $  |x|\lesssim \lambda^{-1}$ and $  |x|\gg  \lambda^{-1}$ separately. 
\subsubsection{Case 1: $  |x|\lesssim  \lambda^{-1}$}
By \eqref{Jm1} and \eqref{Jm3} we have for all $ r\in (1/2, 2)$ the estimate
\begin{equation}
\label{J0derv-est}
\left| \partial_r ^k \left[  \tilde J_{ \frac{d-2}2 }( \lambda r |x|)  \tilde\rho(r) \right]\right| \underset{k}  \lesssim 1  \qquad  ( k\ge 0).
\end{equation}
From Corollary \ref{cor-m} we have 
\begin{equation}\label{mlamb-invest}
\max_{ 1/2 \le r \le 1 }\Bigabs {\partial_r ^k \left(  \frac 1{  m'_{\beta, \lambda}(r)  } \right)} \underset{k}  \lesssim \lambda^{-1}   \angles{ \sqrt \beta \lambda}^{-1} \angles{\lambda}^{\frac12}  \qquad   (k \ge 0).
\end{equation} 

Applying Lemma \ref{lm-corput1} with \eqref{J0derv-est}-\eqref{mlamb-invest} and $N=  d/2$
to \eqref{I-eqq} we obtain
\begin{equation}\label{Iest-2}
\begin{split}
| I_{\lambda} (x, t) |
&\lesssim \lambda^d \cdot \left( \lambda^{-1}   \angles{\sqrt \beta \lambda}^{-1}\angles{\lambda}^{\frac12}   \right)^{  d/2} t^{-   d/2} 
\\
&\lesssim    c_{\beta,d} (\lambda)  t^{-\frac d2} ,
\end{split}
\end{equation}
where to get the second line we used \eqref{tcond}.

\subsubsection{Case 2: $  |x|\gg  \lambda^{-1}$ }
Using \eqref{J0est} in \eqref{I-eqq} we write
\begin{align*}
 I_{\lambda} (x, t) 
 &= \lambda^d  \left[\int_{1/2}^2  e^{it \phi^+_{\lambda} (r)  }  h(\lambda r |x|)  \tilde \rho(r) \, dr +  \int_{1/2}^2  e^{-it \phi^-_{\lambda} (r)  } \bar  h(\lambda r |x|)  \tilde\rho(r) \, dr \right],
\end{align*}
where 
$$
\phi^\pm_{\lambda} (r)=    \lambda r|x|/t  \pm   m_{\beta, \lambda}(r) .
$$
Set $H_{\lambda}( |x|, r) :=h(\lambda r |x|)  \tilde \rho(r)$. In view of \eqref{h-est} we have 
\begin{equation}
\label{Hest}
 \max_{1/2 \le r\le 2}\Bigabs {\partial_r ^k  H_{\lambda}( |x|, r)  }    \lesssim   (\lambda |x|)^{-\frac{d-1}2}  \qquad  ( k\ge 0),
\end{equation}
 where we also used the fact $\lambda |x|\gg 1$.

Now 
we write 
$$
 I_{\lambda} (x, t) 
=  I^+_{\lambda} (x, t) 
+  I^-_{\lambda} (x, t) ,
$$
where
\begin{align*}
 I^+_{\lambda} (x, t) 
 &= \lambda^d  \int_{1/2}^2  e^{it \phi^+_{\lambda} (r)  } H_{\lambda}( |x|, r)  \, dr ,
 \\
I^-_{\lambda} (x, t) &= \lambda^d 
   \int_{1/2}^2  e^{-it \phi^-_{\lambda} (r)  }\bar H_{\lambda}( |x|, r)  \, dr .
\end{align*}
Observe that
$$
\partial_r \phi^\pm_{\lambda} (r)=  \lambda  \left[ |x|/t \pm  m_\beta'(\lambda r) \right],\qquad \partial_r^2\phi^\pm_{\lambda} (r)=     \pm  \lambda^2 m_\beta''(\lambda r),
$$
and hence by Lemma \ref{lm-mest}, 
\begin{equation}
\label{phi'+:est}
|\partial_r \phi^+_{\lambda} (r)|\gtrsim  \lambda\angles{ \sqrt \beta \lambda} \angles{ \lambda }^{-1/2},
\qquad
|\partial^2_r \phi^\pm_{\lambda} (r)| \sim   \lambda^3 \angles{\sqrt \beta \lambda} \angles{   \lambda }^{-5/2}
\end{equation}
for all $ r\in (1/2, 2)$, where we also used the fact that $m_\beta'$ is positive.

\subsubsection*{\underline{Estimate for  $I^+_{\lambda} (x, t)$ } }
 Following as in the proof of  Corollary \ref{cor-m}, we have \begin{equation}
\label{Est-phi+'-1}
\max_{1/2 \le r \le 2}\Bigabs {\partial_r ^k \left( \left[ \partial_r \phi_\lambda^+ (r) \right]^{-1} \right)}  \underset{k} \lesssim  \lambda^{-1}   \angles{\sqrt{\beta} \lambda}^{-1}\angles{\lambda}^{\frac12}   \qquad   (k \ge 0).
\end{equation} 

Applying Lemma \ref{lm-corput1} with \eqref{Hest}, \eqref{Est-phi+'-1}  and $N=  d/2$
to $I^+_{\lambda} (x, t) $ we obtain
\begin{equation}\label{Iest-3}
\begin{split}
| I^+_{\lambda} (x, t) |
&
\lesssim  \lambda^d  \cdot  (\lambda |x|)^{-\frac{d-1}2}  \cdot \left( \lambda^{-1}   \angles{\sqrt{\beta} \lambda}^{-1}\angles{\lambda}^{\frac12}  \right)^{  d/2} t^{- d/2}  
\\
&\lesssim   c_{\beta,d} (\lambda)   t^{-\frac d2} ,
\end{split}
\end{equation}
where to get the second line we used the fact that  $\lambda |x|\gg 1$, and the condition in \eqref{tcond}.

\subsubsection*{\underline{Estimate for $I^-_{\lambda} (x, t)$}}
We treat the the non-stationary and stationary cases separately. In the non-stationary
case, where 
$$ |x | \ll   \angles{\beta \lambda}   \angles{  \lambda }^{-1/2} t \quad \text{or} \quad |x| \gg   \angles{\beta \lambda}  \angles{ \lambda } ^{-1/2} t, $$ we have 
$$
|\partial_r \phi^-_{\lambda} (r)|\gtrsim  \lambda\angles{\sqrt{\beta} \lambda} \angles{  \lambda }^{-1/2},
$$
and hence $I^-_{\lambda} (x, t)$ can be estimated in exactly the same way as $I^+_{\lambda} (x, t)$ above, and satisfies the same bound as in \eqref{Iest-3}.

So it remains to treat the stationary case: 
$$
|x | \sim \angles{\sqrt{\beta} \lambda}  \angles{ \lambda }^{-1/2} t.
$$
 In this case,  
we use Lemma \ref{lm-corput}, \eqref{phi'+:est} and\eqref{Hest} to obtain 
\begin{equation}\label{Iest-station}
\begin{split}
| I^-_{\lambda} (x, t) 
&\lesssim  \lambda^d  \cdot \left(  \lambda^3 \angles{\sqrt{\beta} \lambda}  \angles{   \lambda }^{-5/2}  t \right)^{-\frac12}\left[ |H_\lambda^- (x, 2) |+ \int_{1/2}^2 | \partial_r  H_\lambda^- (x, r)| \, dr\right]
\\
&\lesssim   \lambda^{d-\frac 32 } \angles{\sqrt{\beta} \lambda}^{-\frac12}  \angles{ \lambda}^{\frac54}   t^{-\frac12} \cdot  (\lambda |x|)^{-\frac{d-1}2}
\\
& \lesssim     c_{\beta,d} (\lambda)   t^{-\frac d2}  ,
\end{split}
\end{equation}
where we also used the fact that $H_\lambda^- (x, 2) =0$ and $|x | \sim \angles{ \sqrt{\beta} \lambda} \angles{ \lambda }^{-1/2} t$.

 \section{Proof of Theorem \ref{thm-Str} }

 We shall use the Hardy-Littlewood-Sobolev inequality which
asserts that
\begin{equation}
\label{HLSineq}
\norm{|\cdot |^{-\gamma}\ast f}_{L^a(\R)} \lesssim \ \norm{ f}_{L^b(\R)} 
\end{equation}
whenever $1 < b < a < \infty$ and $0 < \gamma< 1$ obey the scaling condition
$$
\frac1b=\frac1a +1-\gamma.
$$

We prove only \eqref{Strest1d} since \eqref{Strest1d-inh}  follows from \eqref{Strest1d}  by the standard $TT^*$--argument.
First note that \eqref{Strest1d} holds true for the pair $(q, r)=(\infty, 2)$ 
as this is just the energy inequality.  So we may assume $2< q < \infty$ and $r>2$.

Let $q'$ and $r'$ be the conjugates of $q$ and $r$, respectively, i.e., $q'=\frac q{q-1}$ and  $r'=\frac r{r-1}$.
 By the standard $TT^*$--argument, \eqref{Strest1d} is equivalent to the estimate 
\begin{equation}
\label{TTstar}
\norm{ TT^\ast F }_{L^{q}_{t} L^{r}_{ x} (\R^{d+1}) } \lesssim \left[ c_{\beta,d} (\lambda) \right]^{ 1 -\frac 2r } 
\norm{ F  }_{ L^{q'}_{ t} L_x^{r'}(\R^{d+1} )},
\end{equation}
where 
\begin{equation}\label{TTastF}
\begin{split}
 TT^\ast F (x, t)&= \int_{\R^d}  \int_\R e^{i  x  \xi- i(t-s) {m}(  \xi)}  \rho^2_\lambda (\xi)   \widehat{F}( \xi, s)\, ds  d\xi
 \\
 &= \int_\R  K_{\lambda,  t-s} \ast F( \cdot,  s) \, ds
 \end{split}
\end{equation}
with
\begin{align*}
 K_{\lambda ,t}(x)&= \int_{\R^d}  e^{i  x  \xi-it  {m}(  \xi)}  \rho^2_\lambda (\xi)   \, d\xi.
\end{align*}
Since $$K_{\lambda, t} \ast g (x)=e^{ it m(D)}  P_\lambda g_\lambda (x)$$  
 it follows from \eqref{dispest} that
\begin{equation}\label{kest1}
\|K_{\lambda, t} \ast g \|_{L_x^{\infty}(\R^d)} \lesssim c_{\beta,d} (\lambda)  |t|^{-\frac d2}   \|g\|_{L_x^{1}(\R^d)}.
\end{equation}
On the other hand, we have by Plancherel 
\begin{equation}\label{kest2}
\|K_{\lambda, t} \ast g \|_{L_x^{2}(\R^d)} \lesssim   \|g\|_{L_x^{2}(\R^d)}.
\end{equation}
So interpolation between \eqref{kest1} and \eqref{kest2} yields
\begin{equation}\label{kest3}
\|K_{\lambda, t} \ast g \|_{L_x^{r}(\R^d)} \lesssim  \left[ c_{\beta,d} (\lambda)  \right]^{1- \frac2r }  |t|^{- d\left(\frac12-\frac 1r\right)} \|g\|_{L_x^{r'}(\R^d)}
\end{equation}
 for all $  r \in[2, \infty].$

Applying Minkowski's inequality to \eqref{TTastF}, and then use \eqref{kest3} and  \eqref{HLSineq}
with 
$$
(a, b)=(q , q' ), \qquad \gamma= \frac d2-\frac dr=\frac 2q$$
 we obtain
 \begin{align*}
\norm{TT^\ast F }_{L^{q}_{t} L^{r}_{ x} (\R^{d+1})}
&\le \norm{   \int_\R \norm{ K_{\lambda, t-s,} \ast
   F(s, \cdot) }_{L_x^r (\R^d)}  \, ds}_{L^{q}_t(\R)}
  \\
 &\lesssim   \left[ c_{\beta,d} (\lambda)  \right]^{1- \frac2r }   \norm{  \int_\R  |t-s|^{- d\left(\frac12-\frac 1r\right)} 
  \norm{ F(s, \cdot) }_{ L_x^{r'}(\R^d)}  \, ds }_{L_t^{q}(\R)}
   \\
 &\lesssim   \left[ c_{\beta,d} (\lambda)  \right]^{1- \frac2r }  \norm{  
  \norm{ F }_{L_x^{r'}(\R^d) }  }_{L^{q'}_{ t} (\R)}
    \\
 &=    \left[ c_{\beta,d} (\lambda)  \right]^{1- \frac2r } 
  \norm{ F  }_{ L^{q'}_{ t} L_x^{r'}(\R^{d+1})} \, ,
\end{align*}
which is the desired estimate \eqref{TTstar}.

 \section{Proof of Theorem \ref{theorem3d}}

We consider the system \eqref{wtbsq} with $\beta=0$ and a curl-free vector field $\mathbf v $, i.e., $\nabla \times \mathbf v =0 $.
 Observe that 
 $$L_0 (D)= \sqrt { K(D)}, \qquad m_0(D)= |D|\sqrt { K(D)}.$$
 The transformation \eqref{diag-wtsbq} (with $\beta=0$) yields 
$$
\eta=u_++u_-, \quad \mathbf v=-i \sqrt K \mathcal R (u_+- u_-),
$$
where we have used the fact that $\mathbf v $ is curl-free, in which case,
$$
\nabla (\nabla \cdot \mathbf v)= \Delta \mathbf v=-|D|^2 \mathbf v \quad \Rightarrow \quad \mathbf v= -\mathcal R( \mathcal R \cdot \mathbf v).
$$
Consequently,
the Cauchy problem \eqref{wtbsq}, \eqref{data-wtbsq} transforms to
\begin{equation}
\label{wt3d-transf}
\left\{
\begin{aligned}
 (i\partial_t\mp m_0(D) ) u_\pm
&= \mathcal B^\pm (u_+, u_-) ,
\\
u_\pm (0)&= f_\pm,
\end{aligned}
\right.
\end{equation}
where
\begin{equation}
\label{wtnonlin}
\begin{split}
\mathcal B^\pm (u_+, u_-)  &=   2^{-1}|D| K\mathcal R \cdot \left\{   (u_++u_-)    \mathcal R  \sqrt{K}(u_+ - u_-)\right\} 
\\
& \qquad \pm    4^{-1}|D|  \sqrt{K} \left|  \mathcal R \sqrt{K} (u_+ - u_-)\right|^2
\end{split}
\end{equation}
and
\begin{equation}\label{data-trswtsbq}
 f_\pm =\frac{ \sqrt K \eta_0 \mp i  \mathcal  R\cdot \mathbf v_0}{2\sqrt K }  \in H^s( \R^d).
\end{equation}

Thus, Theorem \ref{theorem3d} redues to the following:
\begin{theorem}
\label{theorem3d-reduc}
	Let $d\ge 3$ and $s>\frac d2-\frac 34$. If the initial data has size
	$$
		\sum_\pm \lVert f_\pm \rVert _{H^s}\le \mathcal D_0,
	$$
	then there exists a solution
	$$
		u_\pm \in C \left([0, T];
		H^s(\mathbb{R}^d)\times H^s(\mathbb{R}^d)
		\right)
$$
	of the Cauchy problem \eqref{wt3d-transf}--\eqref{data-trswtsbq} with existence time $T\sim \mathcal D_0^{-2}$.
	
	Moreover, the solution is unique in some subspace of 
	$ C \left([0, T];
		H^s(\mathbb{R}^d)\times H^s(\mathbb{R}^d)
		\right)$ and the solution depends continuously on the initial data.
\end{theorem}

\subsection{Reduction of Theorem \ref{theorem3d-reduc} to bilinear estimates}

The bilinear terms in \eqref{wtnonlin} can be written as 
 \begin{equation*}
\begin{split}
\mathcal B^\pm(u_+, u_-)  &=   \frac12 \sum_{\pm_1, \pm_2} \pm_2 D K \mathcal R \cdot  \left( u_{\pm_1}  \mathcal R   \sqrt{K}u_{\pm_2} \right) 
\\
& \qquad \qquad \pm  \frac14 \sum_{\pm_1, \pm_2}  (\pm_1) (\pm_2) D \sqrt{K }  \left( \mathcal R   \sqrt{K }   u_{\pm_1}   \cdot \mathcal R  \sqrt{K} u_{\pm_2} \right) ,
\end{split}
\end{equation*}
where $\pm_1$ and $\pm_2$ are independent signs.
Then the Duhamel's representation of \eqref{wt3d-transf} is given by
\begin{equation}
\label{uinteq}
\begin{split}
  u_\pm(t) &= \mathcal S_{m_0}(\pm t) f_\pm \mp \frac{ i }2 \sum_{\pm_1, \pm_2}  (\pm_2)\mathcal B_1^\pm (u_{\pm_1},  u_{\pm_2})(t) 
  \\
  & \qquad \qquad \qquad \mp \frac{ i  }4   \sum_{\pm_1, \pm_2}   (\pm_1) (\pm_2) \mathcal B_2^\pm(u_{\pm_1},  u_{\pm_2})(t),
  \end{split}
   \end{equation}
 where $\mathcal S_{m_0}(t)= e^{itm_0(D)}$ and
 \begin{equation}
\label{AB} 
 \begin{split}
\mathcal B_1^\pm (u,v)(t)&:=\int_{0}^{t}    \mathcal S_{m_0}( \pm(t-t') D K \mathcal R \cdot \left( u  \mathcal R \sqrt{K} v \right)(t') \,dt' ,
\\
\mathcal  B_2^\pm (u,v)(t)&:= \int_{0}^{t} \mathcal S_{m_0}( \pm(t-t') D \sqrt{K} \left(   \mathcal R  \sqrt{K} u  \cdot   \mathcal R  \sqrt{K} v \right)(t') \,dt'  .
\end{split}
\end{equation}

\vspace{2mm}

Setting $\beta=0$ in Theorem \ref{thm-Str} we obtain the following.
\begin{corollary}\label{cor-str}
Let $\lambda\in 2^\Z$ and $d\ge 2$. Assume that the pair $(q,r)$ satisfies 
$$
  2< q\le  \infty,  \quad 2\le r \le  \infty, \qquad \frac 2q+ \frac dr = \frac d2.
$$
Then 
 \begin{align}
\label{Strest1d0}
\norm{ \mathcal S_{m_0}(\pm t)  f_{\lambda}}_{ L^{q}_{t} L^{r}_{ x} (\R^{d+1}) } \lesssim  \angles{\lambda}^{\frac{3}{2q}} 
\norm{  f_{\lambda}}_{ L^2_{ x}(\R^d )} ,
\\
\label{Strest1d0-inh}
\norm{ \int_0^t  \mathcal S_{m_0} (\pm(t-s) )F_\lambda (s) \, ds}_{ L^{q}_{t} L^{r}_{ x} (\R^{d+1}) } \lesssim    \angles{\lambda}^{\frac{3}{2q}} 
\norm{ F_{\lambda}}_{ L_t^1L^2_{ x}(\R^{d+1} )} 
\end{align}
for all  $f \in \mathcal{S}(\R^d)$ and  $F \in \mathcal{S}(\R^{d+1})$.

\end{corollary}

Now,
define the contraction space, $X^s_T$ via the norm
$$
\norm{u}_{X^s_T}= \left[ \sum_{\lambda}  \angles{\lambda}^{2s}   \norm{u}^2_{X_\lambda} \right]^\frac 12,
$$
where
$$
\norm{u}_{X_\lambda} =
  \left[ \norm{  P_\lambda u}^2_{ L_T^{\infty}L_x^{2} } + \angles{ \lambda }^{- \frac 3{q}}   \norm{ P_\lambda u}^2_{ L_T^{q} L_x^{r}  }  \right]^\frac12
$$
with 
\begin{equation}\label{qr}
2<q< \infty ,  \qquad  r = \frac {2qd}{qd-4}, \quad d\ge 2. 
\end{equation}
In view of Corollary \ref{cor-str} the pair $(q, r)$ satisfying \eqref{qr}
is an admissible pair.

Observe that
\begin{equation}
\label{Xlam-est}
\norm{ P_\lambda  u}_{ L_T^{\infty}L_x^{2} }  \le \norm{  u}_{X_\lambda} ,
\qquad
\norm{  P_\lambda u}_{L_T^{2} L_x^{r }  }  \le   \angles{ \lambda }^{ \frac 3{2q}}     \norm{  u}_{X_\lambda} .
\end{equation}
Moreover,
$$
  X^s_T \subset L_T^\infty H^s.
$$

 We estimate the linear part of \eqref{uinteq} using 
 \eqref{Strest1d0} as
 \begin{equation}
\label{homest} 
\begin{split}
 \norm{\mathcal S_{m_0}(\pm t)  f_\pm}_{X^s_T} &= \left[ \sum_{\lambda }  \angles{ \lambda}^{2s}   \norm{  \mathcal S_{m_0}(\pm t)  f_\pm }^2_{ X_\lambda }  \right]^\frac 12
 \\
 &\lesssim  \left[ \sum_{\lambda} \angles{ \lambda}^{2s} \norm{  P_\lambda  f_\pm}^2_{ L_x^{2} }  \right]^\frac 12
 \\
 & \sim \norm{   f_\pm}_{ H^{s} } .
\end{split}
\end{equation} 
So Theorem \ref{theorem3d-reduc} reduces to proving bilinear estimates on the terms $ \mathcal B_1^\pm(u,v)$ and $ \mathcal B_2^\pm(u,v) $ that are defined in \eqref{AB}. The proof of the following Lemma will be given in the next section.
\begin{lemma}\label{lm-keybiest}
Let $d\ge 2 \ $, $s>\frac d2-\frac 12- \frac 1{2q}$ for $q>2$, and $T>0$. Then
\begin{align}
\label{Keybiest1}
 \norm{ \mathcal B_1^\pm (u,v) }_{X^s_T} &\lesssim  T^{1-\frac1q}  \norm{u}_{X^s_T}  \norm{v}_{X^s_T}  ,
 \\
 \label{Keybiest2}
 \norm{ \mathcal B_2^\pm (u,v)   }_{X^s_T} &\lesssim  T^{1-\frac1q}   \norm{u}_{X^s_T}  \norm{v}_{X^s_T} 
\end{align}
 for
all $u, \ v \in X^s_T$, where $ \mathcal B_1^\pm$ and $ \mathcal B_2^\pm$ are as in \eqref{AB}.

\end{lemma}

\subsection{Proof of Theorem \ref{theorem3d-reduc} }

Given that Lemma \ref{lm-keybiest} holds,
we solve the integral equations \eqref{uinteq} by contraction mapping techniques. We shall apply  Lemma \ref{lm-keybiest} with
$$ q= \frac 2{1-2\varepsilon}  \quad \text{for} \ \  0 < \varepsilon \ll 1.$$  Consequently, the exponent $s$ will be restricted to
$$
s>\frac d2- \frac34 + \frac \varepsilon 2.
$$

Define the mapping
\[
  (u_+, u_-) \mapsto \left (\Phi^+ (u_+, u_-), \ \Phi^-(u_+, u_-)\right),
\]
where $\Phi^\pm(u_+, u_-)$ is given by the right hand side of \eqref{uinteq}. 
Now given initial data with norm 
$$\sum_\pm \| f_\pm\|_{ H^s}\le \mathcal D_0 , $$ 
we look for a solution in the set
\[
  E_T = \left\{  (u_\pm  \in X^s_T
  \colon  \sum_\pm \| u_\pm\|_{ X^s_T} \le 2 C \mathcal D_0 \right\}.
\]

By \eqref{homest} and Lemma \ref{lm-keybiest} we have
\begin{align*}
\sum_\pm  \norm{\Phi^\pm(u_+, u_-)}_{ X^s_T} &\le C\sum_\pm  \norm{   f_\pm}_{ H^{s} } + C T^{\frac12+ \varepsilon}    \left(\sum_\pm \| u_\pm\|_{ X^s_T}  \right)^2
\\
&\le C \mathcal D_0 + C T^{\frac12+ \varepsilon}    ( 2C \mathcal D_0)^2
\\
&\le 2 C \mathcal D_0,
\end{align*}
provided that 
\begin{equation}\label{T}
T\le \left( 8 C^2 \mathcal D_0\right)^{- \frac 2{1+2\varepsilon}} .
\end{equation}

Similarly, for two pair of solutions $(u_+, u_-)$ and $(v_+, v_-)$ in $E_T$ with the same data, one can derive the difference estimate
\begin{align*}
 &\sum_\pm  \norm{ \Phi^\pm(u_+, u_-)- \Phi^\pm (v_+, v_-)}_{ X^s_T} 
 \\
 & \quad \le C T^{\frac12+ \varepsilon}   \left(\sum_\pm \| u_\pm\|_{ X^s_T}+ \| v_\pm\|_{ X^s_T}  \right) \left(\sum_\pm \| u_\pm- v_\pm\|_{ X^s_T}  \right)
  \\
 & \quad \le 4 C^2 T^{\frac12+ \varepsilon}   \mathcal D_0 \left(\sum_\pm \| u_\pm- v_\pm\|_{ X^s_T}  \right)
 \\
 & \quad \le \frac12   \left(\sum_\pm \| u_\pm- v_\pm\|_{ X^s_T}  \right),
\end{align*}
where in the last inequality we used \eqref{T}.

Therefore, $( \Phi^+, \Phi^-)$ is a contraction on $E_T$ and therefore it has a unique fixed point $(u_+, u_-) \in E_T$ solving the integral equation \eqref{uinteq}--\eqref{AB} on $\R^d\times [0, T]$, where $T \sim  \mathcal D_0^{-2-}$. Uniqueness in the space $X^s_T \times X^s_T$ and continuous dependence on the initial data can be shown in a similar way, by the difference estimates. This concludes the proof of Theorems \ref{theorem3d-reduc}.

\section{Proof of Lemma \ref{lm-keybiest} }

First we prove some bilinear estimates in Lemma \ref{lm-biest} below that will be crucial in the proof of Lemma  \ref{lm-keybiest}. To do so, we need the following Bernstein inequality, which is valid for $1\le a \le b \le \infty$
(see, for instance, \cite[Appendix A]{Tao}).
\begin{equation}\label{Bern-est0}
 \|  f_{\lambda} \|_{L^b_x}   \lesssim \lambda^{\frac da -\frac db} \|f_{\lambda}\|_{L^a_x},
\end{equation}
Moreover, we have for all $s_1, s_2 \in \R$ and $p\ge 1$,
\begin{equation}\label{Bern-est}
 \| |D|^{s_1} K^{s_2} \mathcal R f_{\lambda} \|_{L^p_x}   \lesssim  \lambda^{s_1} \angles{  \lambda}^{-s_2}  \|f_{\lambda}\|_{L^p_x},
\end{equation}
where we used $K(\xi) \sim  \angles{ \xi}^{-1} $.

\begin{lemma}
\label{lm-biest}
Let $d\ge 2$,  $q>2$,  $T>0$ and  $\lambda_j \in2^\Z$ ($j=0,1, 2)$. Then 
\begin{align}
\label{keybiles11}
\norm{ |D| K   P_{\lambda_0} 
 \mathcal R \cdot \left(   u_{\lambda_1} \mathcal R \sqrt  K v_{\lambda_2} \right) }_{L^1_T L^2_x}  
& \lesssim   \frac{  T^{1-\frac1q} \min\left( \angles{ \lambda_1},  \angles{ \lambda_2} \right)^{\frac d{2}-\frac1{2q} }}{ \angles{ \lambda_2} ^{\frac12} }
\norm{u}_{  X_{\lambda_1 } }
\norm{v}_{  X_{\lambda_2 } },
\\
\label{keybiles12}
\norm{ |D| \sqrt {K   }
 P_{\lambda_0}\left(   \mathcal R  \sqrt  K u_{\lambda_1}  \cdot \mathcal R \sqrt  K v_{\lambda_2} \right) }_{L^1_T L^2_x}  
& \lesssim   \frac{ T^{1-\frac1q}  \angles{ \lambda_0} ^{\frac12}
 \min\left( \angles{ \lambda_1},  \angles{ \lambda_2} \right)^{\frac d{2}-\frac1{2q} } }{\angles{ \lambda_1} ^{\frac12} \angles{ \lambda_2} ^{\frac12}}
\norm{u}_{  X_{\lambda_1 } }
\norm{v}_{  X_{\lambda_2 } }
\end{align}
for all $u\in X_{\lambda_1 }$ and $v \in X_{\lambda_2 }$.
\end{lemma}
\begin{proof} 

We only prove \eqref{keybiles11} since the proof for \eqref{keybiles12} is similar.
To prove \eqref{keybiles11} ,
by symmetry, we may assume $\lambda_1 \lesssim \lambda_2$. 
Let
$$
q= 2^+,  \qquad  r = \frac {2qd}{qd-4},  \qquad  d\ge 2 
$$
Then by H\"{o}lder,  \eqref{Bern-est},  \eqref{Bern-est0} and \eqref{Xlam-est} we obtain
\begin{align*}
\text{LHS} \ \eqref{keybiles11} & \lesssim T^\frac12 \lambda_0\angles{ \lambda_0}^{-1} \norm{   \mathcal R \cdot( u_{\lambda_1} \mathcal R \sqrt  K v_{\lambda_2} ) }_{L_{T,x}^2}  
\\
&\lesssim  T^{\frac12}  
 \angles{ \lambda_2}^{-\frac12}   \norm{u_{\lambda_1}}_{ L_T^2 L_x^\infty } \norm{  v_{\lambda_2}}_{ L_T^\infty L_x^2 }  
 \\
&\lesssim  T^{1-\frac1q}    
 \angles{ \lambda_2}^{-\frac12}  \lambda_1^{\frac d2-\frac2q} \norm{u_{\lambda_1}}_{ L_T^q L_x^r } \norm{  v_{\lambda_2}}_{ L_T^\infty L_x^2 }  
\\
&\lesssim   T^{1-\frac1q}    
 \angles{ \lambda_2}^{-\frac12}  \lambda_1^{\frac d2-\frac2q}  \angles{ \lambda_1}^{\frac 3{2q}} 
   \norm{u}_{X_{\lambda_1} }
  \norm{ v }_{X_{\lambda_2}}
  \\
&\lesssim   T^\frac12  
 \angles{ \lambda_2}^{-\frac12}    \angles{ \lambda_1}^{\frac d2 -\frac 1{2q}} 
   \norm{u}_{X_{\lambda_1} }
  \norm{ v }_{X_{\lambda_2}}
\end{align*}
which proves \eqref{keybiles11}.
\end{proof}

Now we are ready to prove Lemma  \ref{lm-keybiest}. To this end, we decompose 
 $ u=\sum_{\lambda} u_{\lambda}$ and $v=\sum_{\lambda} v_{\lambda}.$
Note that by denoting
$$
a_{\lambda}:=  \norm{u }_{  X_{\lambda} }, \quad  b_{\lambda}:=  \norm{v}_{  X_{\lambda }} $$
we can write
\begin{equation}\label{Xs-redf}
\norm{u}_{  X^s_T }= \norm{ \left(  \angles{\lambda}^{s}  a_{\lambda } \right)}_{l^2_{\lambda}}, \quad
\norm{v}_{  X^s_T }= \norm{ \left(  \angles{\lambda}^{s}  b_{\lambda } \right)}_{l^2_{\lambda}}.
\end{equation}

We shall make a frequent use of of the following dyadic summation estimate, for $\mu, \lambda \in 2^\Z$ and $ c_1, c_2,  p>0$:
\begin{equation*}
\sum_{\mu \sim \lambda } a_\mu \sim a_\lambda, \qquad 
\sum_{c_1 \lesssim \lambda \lesssim c_2} \lambda^p \lesssim 
\begin{cases}
& c_2^p  \quad \text{if} \ p>0,
\\
& c_1^p  \quad \text{if} \ p<0.
\end{cases}
\end{equation*} 

\subsection{Proof of \eqref{Keybiest1} }
Applying \eqref{Strest1d0-inh} to $\mathcal B_1^\pm (u,v)$ in \eqref{AB} we get  
  \begin{align*}
\norm{   \mathcal B_1^\pm  (u,v)  }_{X^s_T} ^2 
& \lesssim \sum_{\lambda_0}  \angles{\lambda_0}^{2s} \norm{   |D| K    \mathcal R \cdot P_{\lambda_0}  \left( u  \mathcal R \sqrt{K} v \right) }^2_{ L_T^1L_x^{2} } ,
 \end{align*}
 where \eqref{Xlam-est} is also used.
 By the dyadic decomposition 
 \begin{equation}\label{Iloc}
 \norm{   |D| K    \mathcal R \cdot P_{\lambda_0}  \left( u  \mathcal R \sqrt{K} v \right) }_{  L_T^1L_x^{2}  } 
 \lesssim \sum_{\lambda_1, \lambda_2}   \norm{   |D| K   \mathcal R \cdot P_{\lambda_0}  \left( u_{\lambda_1}  \mathcal R \sqrt{K} v_{\lambda_2}  \right) }_{ L_T^1L_x^{2}   } .
 \end{equation}
Now let $\lambda_\text{min}  $, $\lambda_\text{med}$  and $ \lambda_{\max}$ denote the minimum, median and the maximum of $\{\lambda_0, \lambda_1, \lambda_2\}$, respectively.
 By checking the support properties in Fourier space of the bilinear term on the right hand side of \eqref{Iloc} one can see that this term vanishes unless $\bm \lambda=(\lambda_0, \lambda_1, \lambda_2) \in   \Lambda,$
where
\begin{equation*}
  \Lambda= \{\bm \lambda : \  \lambda_\text{med} \sim \lambda_{\max}  \}.
  \end{equation*} 
Thus, we have a non-trivial contribution in \eqref{Iloc} only if $\bm \lambda \in \cup_{j=0}^2  \Lambda_j,$ where
\begin{align*}
\Lambda_0 &= \{\bm \lambda : \ \lambda_0 \lesssim \lambda_1\sim  \lambda_2 \}, 
\\ 
\Lambda_1&= \{\bm \lambda : \ \lambda_2 \ll  \lambda_1\sim  \lambda_0 \}, 
\\ 
\Lambda_2 &= \{\bm \lambda : \ \lambda_1 \ll \lambda_2 \sim  \lambda_0 \}.
\end{align*}
 By using these facts, and applying \eqref{keybiles11} to the right hand side of \eqref{Iloc}, we get 
 \begin{align*}
\norm{   \mathcal B^\pm_1 (u,v)  }_{X^s_T} ^2 
& \lesssim T^{2-\frac2q}  \sum_{j=0}^2   \mathcal I_j
 \end{align*}
 where 
\begin{equation}
\label{Gammadefj}
  \mathcal I_j= \sum_{\lambda_0}  \angles{\lambda_0}^{2s} \left[  \sum_{\lambda_1, \lambda_2: \ \bm \lambda \in   \Lambda_j }   \frac{  \min\left( \angles{ \lambda_1},  \angles{ \lambda_2} \right)^{\frac d{2}-\frac1{2q} }}{ \angles{ \lambda_2} ^{\frac12} } a_{\lambda_1 } b_{\lambda_2 }    \right]^2
\end{equation} 
So \eqref{Keybiest1} reduces to proving 
\begin{equation}
\label{GammaEstj}
 \mathcal I_j \lesssim \norm{u}^2_{  X^s_T }\norm{v}^2_{ X^s_T }  \qquad \text{if} \ \  s>\frac d2 - \frac12 - \frac 1{2q} 
\end{equation} 
for $j=0,1,2$. 

These are shown to hold as follows:
\begin{align*}
 \mathcal I_0
&
\lesssim  \sum_{ \lambda_0 } \angles{\lambda_0}^{2s} \left( \sum_{ \lambda_1\sim \lambda_2\gtrsim \lambda_0  }  
 a_{\lambda_1 }  \cdot \angles{\lambda_2}^{\frac d 2-\frac12-\frac1{2q}} b_{\lambda_2 } \right)^2
 \\
 &
\lesssim  \sum_{ \lambda_0 } \angles{\lambda_0}^{-2\left( s- \frac d2 + \frac12 +\frac1{2q}\right)} \left( \sum_{ \lambda_1\sim \lambda_2\gtrsim \lambda_0  }  
\angles{\lambda_1}^{s}  a_{\lambda_1 }  \cdot \angles{\lambda_2}^{s} b_{\lambda_2 } \right)^2
\\
&
\lesssim   \norm{u}^2_{  X^s_T }\norm{v}^2_{ X^s_T },
\end{align*}
where to obtain the last line we used Cauchy Schwarz inequality in $\lambda_1\sim \lambda_2$ and \eqref{Xs-redf}.

Similarly,
\begin{align*}
\mathcal I_1
&
\lesssim \sum_{ \lambda_0 }   \angles{\lambda_0}^{2s} \left( \sum_{ \lambda_2 \ll \lambda_1 \sim \lambda_0}  
a_{\lambda_1 } \cdot  \angles{\lambda_2}^{\frac d 2-\frac12-\frac1{2q}}  b_{\lambda_2 } \right)^2
\\
& \lesssim   \sum_{ \lambda_0 }   \angles{\lambda_0}^{2s} a_{\lambda_0}^2 \left( \sum_{ \lambda_2  } \angles{\lambda_2}^{\frac d 2-\frac12-\frac1{2q}}  b_{\lambda_2 } \right)^2
\\
&
\lesssim \norm{u}^2_{  X^s_T }\norm{v}^2_{ X^s_T },
\end{align*}
where to get the last two inequalities we used 
$\sum_{\lambda_1 \sim \lambda_0} a_{\lambda_1} \sim a_{\lambda_0} $ and by Cauchy Schwarz
\begin{align*}
 \sum_{ \lambda_2  }    \angles{\lambda_2}^{\frac d 2-\frac12-\frac1{2q}} b_{\lambda_2 } &= \sum_{ \lambda_2  }   \angles{\lambda_2}^{\frac d 2-\frac12-\frac1{2q} -s}  \cdot \angles{  \lambda_2}^{s} b_{\lambda_2 }  \\
  &\lesssim  \ \norm{ \left(  \angles{\lambda_2}^{s}  b_{\lambda_2 } \right)}_{l^2_{\lambda_2}} \lesssim \norm{v}_{ X^s_T }.
\end{align*}

Finally, 
\begin{align*}
\mathcal I_2
&
\lesssim \sum_{ \lambda_0 }   \angles{\lambda_0}^{2s} \left( \sum_{ \lambda_1 \ll \lambda_2 \sim \lambda_0}  \angles{  \lambda_1}^{\frac d 2-\frac12-\frac1{2q}} 
a_{\lambda_1 } \cdot  b_{\lambda_2 } \right)^2
\\
& \lesssim   \sum_{ \lambda_0 }   \angles{\lambda_0}^{2s} b_{\lambda_0}^2 \left( \sum_{ \lambda_1  }  
   \angles{  \lambda_1}^{\frac d 2-\frac12-\frac1{2q}}  a_{\lambda_1 } \right)^2
\\
&
\lesssim \norm{u}^2_{  X^s_T }\norm{v}^2_{ X^s_T }.
\end{align*}

\subsection{Proof of \eqref{Keybiest2} }
Arguing as in the preseding subsection we apply \eqref{Strest1d0-inh} to $\mathcal B_2^\pm (u,v)$ in \eqref{AB}  and then use \eqref{keybiles12} to obtain
 \begin{align*}
\norm{   \mathcal B^\pm_2 (u,v)  }_{X^s_T} ^2 
& \lesssim T^{2-\frac2q} \sum_{j=0}^2  
 \widetilde {\mathcal I}_j
 \end{align*}
 where 
\begin{equation*}
  \widetilde {\mathcal I}_j= \sum_{\lambda_0}  \angles{\lambda_0}^{2s} \left[  \sum_{\lambda_1, \lambda_2 \in \Lambda_j}   C(\lambda_0, \lambda_1, \lambda_2) a_{\lambda_1 } b_{\lambda_2 }    \right]^2
 \end{equation*}
 with
 $$
 C(\lambda_0, \lambda_1, \lambda_2)=\frac{   \angles{ \lambda_0} ^{\frac12}
 \min\left( \angles{ \lambda_1},  \angles{ \lambda_2} \right)^{\frac d{2}-\frac1{2q} } }{\angles{ \lambda_1} ^{\frac12} \angles{ \lambda_2}. ^{\frac12}}
 $$
 
 So \eqref{Keybiest2} reduces to proving 
\begin{equation*}
 \widetilde { \mathcal I}_j \lesssim \norm{u}^2_{  X^s_T }\norm{v}^2_{ X^s_T } \quad (j=0,1,2).
\end{equation*} 
These can be proved following the argument of the preceding subsection by using the fact that
\begin{equation*}
  C(\lambda_0, \lambda_1, \lambda_2) \lesssim \min\left (  \angles{  \lambda_1} ,   \angles{  \lambda_2}\right)^{\frac d2-\frac12-\frac 1{2q}}  
\end{equation*} 
for all $\lambda_0, \lambda_1, \lambda_2 \in \Lambda$.

\section{Appendix}

In this appendix, we derive some useful estimates on the derivatives of all order for the function 
$$
m_\beta (r)=\sqrt{ r\left(1+ \beta r^2\right)  \tanh(  r ) },   \qquad  \beta \in \{0, 1\}.
$$
 Estimates for the first and second order derivatives of this function is derived recently in \cite{PSST21}.

Clearly,
\begin{equation}
\label{m-est}
m_\beta(r)  \sim r  \angles{ \sqrt{\beta} r} \angles{  r}^{-1/2}.
\end{equation}

\begin{lemma}\label{lm-mest}
Let  $\beta \in \{0, 1\}$ and $r>0$. Then
\begin{align}
\label{m1-est} 
m_\beta'(r) & \sim \angles{ \sqrt{\beta} r} \angles{ r}^{-1/2},
\\
\label{m2-est} 
|m_\beta''(r)| &\sim  r \angles{ \sqrt{\beta}  r} \angles{ r}^{-5/2}.
\end{align}
Moreover, 
\begin{align}
\label{mj-est} 
|m_\beta^{(k)}(r)|\underset{k}  \lesssim r^{1- k}  \angles{ \sqrt{\beta}  r} \angles{ r}^{-1/2} \qquad ( k \ge 3 ).
\end{align}

\begin{proof}
The estimates \eqref{m1-est} and \eqref{m2-est} are proved in \cite[Lemma 3.2, see its proof in Section 5]{PSST21}. So we only prove \eqref{mj-est}. 

Let
$$
T(r)=\tanh r, \qquad  S(r)=\sech  r.
$$
Then 
\begin{align*}
T'&=S^2, \qquad S'=-T S,
\qquad
T''=-2TS^2.
\end{align*}
In general,  we have
$$
T^{(j)} (r)= S^2 \cdot P_{j-1}( S, T)  \qquad  (j\ge 1)
$$
for some polynomial $P_{j-1}$ of degree $j-1$.

Clearly,
\begin{equation*}
T(r) \sim  r\angles{ r}^{-1} \qquad  \text{and} \qquad S(r) \sim e^{-r}. 
\end{equation*}
So $ |P_{j-1}( S, T) |\lesssim 1$, and hence
\begin{equation}\label{Tj-est}
|T^{(j)}(r)|  \lesssim e^{-2r}   \qquad (j\ge 1).
\end{equation}
 
 Write
 $$
m_\beta(r)=   f_\beta(r)  \cdot T_0(r),
$$
where 
$f_\beta(r)= \sqrt{ r} \angles{ \sqrt{\beta}  r}$ and  $ T_0(r)=\sqrt { T(r) }  .$  
One can show that
\begin{equation}\label{rhoj-est-est}
\Bigabs{f_\beta^{(j)}(r)}  \lesssim  r^{\frac12 -j} \angles{ \sqrt{\beta}  r} \qquad ( j\ge 0).
\end{equation}

Combining \eqref{Tj-est} with $T(r) \sim  r\angles{ r}^{-1} $ we obtain
  \begin{equation}\label{sqrtTj-est}
T_0(r) \sim  r^\frac12 \angles{ r}^{-\frac12} , \qquad \Bigabs{ T^{(j)}_0(r) }  \lesssim  r^{\frac12 -j }  \angles{r}^{j-\frac12 } e^{-2
r}  \qquad ( j\ge 1).
\end{equation}

Finally,  we use \eqref{rhoj-est-est} and  \eqref{sqrtTj-est} 
to obtain for all $k \ge 3$,
\begin{align*}
\Bigabs{m_\beta^{(k)}(r)} &=  \Bigabs{ f_\beta^{(k)}(r)  T_0(r)  +  \sum_{j=1}^k 
\begin{pmatrix}
k \\ j
\end{pmatrix}    f_\beta^{(k-j)}(r)  T^{(j)}_0(r)}
\\
&\lesssim   r^{1 -k} \angles{ \sqrt{\beta}  r} \angles{ r}^{-\frac12}  + \sum_{j=1}^k 
\begin{pmatrix}
k \\ j
\end{pmatrix}    r^{\frac12 -(k-j)} \angles{ \sqrt{\beta}  r} \cdot r^{\frac12 -j }  \angles{r}^{j-\frac12 } e^{-2r}
\\
&\lesssim   r^{1 -k}  \angles{ \sqrt{\beta} r} \angles{ r}^{-\frac12} .
\end{align*}

\end{proof}

\end{lemma}

\begin{corollary}\label{cor-m}
For $\lambda, r>0$, define
 $m_{\beta, \lambda}(r)= m_\beta(\lambda r)$.  Then 
\begin{equation}\label{m-invest}
\max_{ r\sim 1}\Bigabs {\partial_r ^k \left(  \frac 1{  m'_{\beta, \lambda}(r)} \right)} \underset{k}  \lesssim \lambda^{-1}   \angles{ \sqrt{\beta}  \lambda}^{-1}\angles{\lambda}^{\frac12}   \qquad   (k \ge 0).
\end{equation} 
\begin{proof}
Observe that
$$
m^{(k)}_{\beta, \lambda}(r) = \lambda^k m_\beta^{(k)}(\lambda r)   \qquad (k\ge 1).
$$
By \eqref{m1-est} we have, for $r\sim 1$,
$$
| m'_{\beta, \lambda}(r) |\sim  \lambda  \angles{\sqrt{\beta}  \lambda} \angles{\lambda}^{-\frac12}
$$
and by \eqref{m2-est}-- \eqref{mj-est} 
$$
|m^{(k)}_{\beta, \lambda}(r)   | \underset{k} \lesssim   \lambda  \angles{\sqrt{\beta}  \lambda} \angles{\lambda}^{-\frac12} \qquad (k\ge 2).
$$
Finally, one can combine 
these two estimates with the differentiation formula
\begin{align*}
\partial_r^k \left( \frac 1f \right) =\sum_{p=1}^k \sum_{ \substack{k_1, \cdots, k_p \in \N \\  k_1+ \cdots + k_p =k  }} c_{p, k_1, \cdots, k_p } \frac{ \partial_r^{k_1} f \cdots \partial_r^{k_p} f  }{f^{p+1}}
\end{align*}
to obtain the desired estimate \eqref{m-invest}.

\end{proof}

\end{corollary}


\noindent \textbf{Acknowledgments}
The authors would like to thank the anonymous referee for useful comments on an earlier version of this article.



\end{document}